\documentclass[11pt]{amsart}
\usepackage{graphicx}
\usepackage{amsfonts, amsmath, amsfonts, amssymb}
\newtheorem{thm}{Theorem}[section]

\newtheorem{lem}[thm]{Lemma}

\theoremstyle{definition}
\newtheorem{defn}[thm]{Definition}
\theoremstyle{remark}
\newtheorem{rem}[thm]{Remark}

\newenvironment{prf1}{\noindent{\it Proof 1}}{\\ \hspace*{\fill}$\Box$ \par}
\newenvironment{prf2}{\noindent{\it Proof 2}}{\\ \hspace*{\fill}$\Box$ \par}

\begin{document}

\title[MetricRicci]{Metric Ricci curvature for $PL$ manifolds}

\author[Emil Saucan]{Emil Saucan}

\address{Department of Mathematics, Technion, Haifa, Israel}

\email{semil@tx.technion.ac.il}

\thanks{Research supported 
by European Research Council under the European Community's Seventh Framework Programme (FP7/2007-2013) / ERC grant agreement n${\rm ^o}$ [203134].}
\subjclass{51K10, 53C21, 65D18}%
\keywords{$PL$ manifold, Ricci curvature, Bonnet-Myers Theorem, Wald-Berestovskii curvature}
\date{\today}

\maketitle


\begin{abstract}
We introduce a metric notion of Ricci curvature for $PL$ manifolds and study its convergence properties. We also prove a fitting version of the Bonnet-Myers Theorem, for surfaces as well as for a large class of higher dimensional manifolds.
\end{abstract}


\section{Introduction}

Recent years have seen a great ``revival'' of Ricci curvature, due mainly to Perelman's 
celebrated work on the the Ricci flow and the Poincare conjecture \cite{Per}, \cite{Per1}, but also to its extension to a far larger class of geometric objects, than merely smooth (3-)manifolds (see \cite{Vi} and the bibliography therein). In consequence, Ricci curvature has become an object of interest and study in Graphics and Imaging. The approaches range from implementations of the Combinatorial Ricci curvature of Chow and Luo \cite{CL} -- see, e.g. \cite{GY}, through classical approximation methods of smooth differential operators \cite{As}, \cite{FAF}, to discrete, purely combinatorial methods 
\cite{SAWZ1}.

We have addressed the problem of Ricci curvature of $PL$ surfaces and higher dimensional $PL$ (piecewise flat) manifolds, from a metric point of view, both as a tool in studying the Combinatorial Ricci flow on surfaces \cite{Sa11}, and, in a more general context, in the approximation in secant of curvature measures of manifolds \cite{Sa11a}. Computational applications aside, these and related problems -- see \cite{Ru}, \cite{Bernig} -- make the study of a robust notion of Ricci curvature for $PL$ spaces a subject of thriving interest in the Geometry and Topology of (mainly 3-dimensional) manifolds.

This article represents a continuation
of (or, in a sense, an appendix to) both papers above: In the first part of this paper we address the main problem, namely that of defining a computable, discrete metric Ricci curvature for $PL$ (piecewise flat) manifolds. Here we use methods similar to those in \cite{Sa11a}, to address a problem -- or, rather, a particular case -- considered therein.
We also investigate the convergence properties of this newly introduced curvature.
In the second (and last) part we address a rather more theoretical question, namely if the newly introduced version of Ricci curvature satisfies -- as indeed expected from a proper (``correct'') notion of Ricci curvature -- a Bonnet-Myers type of theorem. (The methods in this part are, partially, those developed in \cite{Sa11}.)

A note to the reader before we proceed to the main part of our paper: Our default source for geometric differential definitions and results is \cite{Be}, and, if no other source is specified the reader should consult, if needed, this encyclopedic source. 
Also, as background material for $PL$ topology we refer the reader to \cite{Hu}.

\section{
Definition and Convergence}

     To begin with, we have to be able to properly define Ricci curvature for $PL$ manifolds. This is indeed possible, not just for $PL$ manifolds but also for polyhedral ones -- and in a quite natural manner -- combining ideas of Stone \cite{St1}, \cite{St2} and metric curvatures.
      For this one regards 
      Ricci curvature as the mean of sectional curvatures:
      %
%
\begin{equation}
{\rm Ric}(e_1) = {\rm Ric}(e_1,e_1) = \sum_{i=2}^nK(e_1,e_i)\,,
\end{equation}
for any orthonormal basis $\{e_1,\cdots,e_n\}$, and where $K(e_1,e_j)$ denotes the sectional curvature of the 2-sections containing the directions $e_1$.


      %
First, one has, of course, to be able to define (variational) {\it Jacobi fields}. 
This is where we rely upon Stones's work. However, we do not need the whole force of this technical apparatus, only to determine the relevant two sections and, of course, to decide what a direction at a vertex of a $PL$ manifold is.

In fact, in Stone's work, combinatorial Ricci curvature is defined both for the given simplicial complex $\mathcal{T}$, and also for its {\it dual complex} $\mathcal{T}^\ast$. In the later case, cells -- playing here the role of the planes in the classical setting of which sectional curvatures are to be averaged -- are considered. However, his approach for the given complex, where one computes the Ricci curvature ${\rm Ric}(\sigma,\tau_1-\tau_2)$ of an $n$-simplex $\sigma$ in the direction of two adjacent $(n-1)$-faces, $\tau_1,\tau_2$, is not natural in a geometric context (even if useful in his purely combinatorial one), except for the 2-dimensional case, where it coincides with the notion of Ricci curvature in a direction (i.e., in this case, an edge -- see also Remark \ref{rem:defect} below).
Passing to the dual complex will not restrict us, since $(\mathcal{T}^\ast)^\ast = \mathcal{T}$ and, moreover -- and more importantly -- considering {\it thick} triangulations enables us to compute the more natural metric curvature for the dual complex and use the fact that the dual of a thick triangulation is thick, as we shall 
detail below. Working only with thick triangulations does not restrict us, however, at least in dimension $\leq 4$, since any triangulation admits a ``thickening'' -- see \cite{Sa05}.\footnote{This holds, as already mentioned, for any $PL$ manifold of dimension $\leq 4$, and in all dimensions for smoothable $PL$ manifolds, as well for any manifold of class $\geq \mathcal{C}^1$. Since the proof of the main result of Section 3, regarding manifolds of dimension higher than 3, holds only for manifolds admitting smoothings, restricting ourselves only to such manifolds does not represent anyhow a great hindrance.}

First, let us recall the definition of thick triangulations:

\begin{defn} Let $\tau \subset \mathbb{R}^n$ ; $0 \leq k \leq n$ be a $k$-dimensional simplex.
The {\it thickness} (or {\it fatness}) $\varphi$ of $\tau$ is defined as being:
\begin{equation} \label{eq:fat-Munkres}
\varphi(\tau) = \frac{{\rm dist}(b,\partial\sigma)}{{\rm diam}\,\sigma}\,,
\end{equation}
where $b$ denotes the {\it barycenter} of $\sigma$ and $\partial\sigma$ represents the standard notation for the {\it boundary} of $\sigma$ (i.e the union of the $(n-1)$-dimensional faces of $\sigma$).

A simplex $\tau$ is $\varphi_0${\it-thick}, for some $\varphi_0 > 0$, if $\varphi(\tau) \geq \varphi_0$. A triangulation (of a submanifold of $\mathbb{R}^n$) $\mathcal{T} = \{ \sigma_i \}_{i\in \bf I}$ is
$\varphi_0${\it-thick} if all its simplices are $\varphi_0$-thick. A
triangulation $\mathcal{T} = \{ \sigma_i \}_{i\in \bf I }$ is {\it
thick} if there exists $\varphi_0 \geq 0$ such that all its
simplices are $\varphi_0${\it-thick}.
\end{defn}

\begin{rem}
This is Munkres' definition \cite{Mun}. For a discussion of other, equivalent definitions, their mutual interplay and relationship with certain aspects of Differential Geometry (mainly curvature approximation) see \cite{Sa11a}. Note that this definition holds for more general simplices, not necessarily Euclidean ones.
\end{rem}

%
%
%
%
%

To be able to define and estimate the Ricci curvature of $\mathcal{T}$ and $\mathcal{T^*}$ and the connection between them, we have to make appeal in an essential manner to the fatness of the given complex.
We begin by noting -- by
keeping in mind formula (\ref{eq:fat-Munkres}) -- that, since the length of the edge $l_{ij}^*$, dual to the edge $l_{ij}$ common to the faces $f_i, f_j$ equals $r_i + r_j$, the first barycentric subdivision\footnote{needed in the construction of the dual complex -- see e.g. \cite{Hu}} of a thick triangulation is thick.\footnote{For planar triangulations, and also for higher dimensional complexes (embedded in some $\mathbb{R}^N$), one can realize the dual complex (also in $\mathbb{R}^N$) by constructing the dual edges $l_{ij}^*$ orthogonal to the middle of the respective $l_{ij}$-s. To show the thickness of the dual simplices, one has also to make appeal to the characterization of thickness in terms of dihedral angles (Conditions (1.15) of \cite{cms}).}
%
Since the definition of thickness  also makes sense for for general cells (see \cite{SAZ}, Definition 3.1), we can summarize the discussion above as

\begin{lem}
The dual complex of a thick (simplicial) complex is thick.
\end{lem}


Moreover, we have the following (common) Gromov-Hausdorff convergence property:


\begin{lem} \label{ref:T,T*}
Let $\mathcal{T},\mathcal{T^*}$ be as above. Then
\begin{equation}
\lim_{\delta(\mathcal{T}) \rightarrow  0}(\mathcal{T}) = \lim_{\delta(\mathcal{T}) \rightarrow 0}(\mathcal{T^*})\,,
\end{equation}
where $\delta(\mathcal{T}), \delta(\mathcal{T^*})$ denote the mesh of $\mathcal{T},\mathcal{T^*}$, respectively.
\end{lem}

\begin{rem}
It is important to stress here the crucial role of the thickness of the triangulation, as far as geometry is concerned: Thickness ensures, by its definition, the fact that no degeneracy of the simplices occurs, hence no collapse and degeneracy of the metric can take place.
Moreover, 
in its absence no uniform estimates for the edge lengths can be made, hence convergence of (dual) meshes and, as we shall see shortly, of their metric Ricci curvatures, can no be guaranteed.
\end{rem}

Returning to the definition of Ricci curvature for simplicial complexes: Given a vertex $v_0$, in the dual of a $n$ dimensional simplicial complex, a {\it direction} at $v_0$ is just an oriented edge $e_1 = v_0v_1$. Since, 
there exist precisely $n$ 2-cells, $\mathfrak{c}_1,\ldots,\mathfrak{c}_{n}$\,, having $e_1$ as an edge and, moreover, these cells form part of $n$ relevant variational (Jacobi) fields, the Ricci curvature at the vertex $v$, in the direction $e_1$ is simply
\begin{equation} \label{eq:RicciCell}
{\rm Ric}(v) = \sum_{i=1}^nK(\mathfrak{c}_i)\,.
\end{equation}

\begin{rem}
Observe that the index ``$i$'' in the definition (\ref{eq:RicciCell}) above runs from $1$, and not from $2$, as expected judging from the classical (smooth) setting.  This is due to the fact that we defined Ricci curvature by passing to the dual complex, with its simple but demanding (so to say) combinatorics. For the implications of this fact, see Theorem \ref{thm:ComparisonThm} and Remark \ref{rem:ComparisonThm} below.
\end{rem}

\begin{rem}
Note that we followed \cite{St1} only in determining the variational fields, but not in his definition of Ricci curvature. Indeed, he considers a direction at a vertex $v_0$ to be the union of two edges $e_1,e_2$ in the dual complex, where $e_1 = (v_0,v_1), e_2 = (v_1,v_2)$ and the direction is determined by the lexicographical order. Then (according to \cite{St1}) the relevant variational field are given by the $2n$ distinct 2-cells $\mathfrak{c}_1,\ldots,\mathfrak{c}_{2n}$, containing the edges $e_1$ and $e_2$, $2n-1$ of them containing one, but not both of them. Hence, the Ricci curvature at $v$ in the direction $e_1e_2$ is to be taken as the total defect of these $2n-1$ cells. This approach is necessary
in the combinatorial case. However, it is more difficult than our approach and it would produce unnecessary complications in determining the relevant analogues of the $(n-1)$ 2-sections of the classical, smooth case. Moreover, it is quite possible that, in any practical implementation, the advantages obtained by considering larger variational fields would be countermined by ``noise'' added by considering such order 2 (or larger) neigbourhoods of the given vertex. However, computing Ricci curvature according to this scheme is still possible, using our metric approach (but see also the following Remark \ref{rem:defect}).
\end{rem}

\begin{rem} \label{rem:defect}
It is still possible to compute Ricci curvature according, more-or-less, to Stone's ideas, at least for the 2-dimensional case. Indeed, according to \cite{St2}
\begin{equation}
{\rm Ric}(\sigma,\tau_1-\tau_2) = 8n - \sum_{j=1}^{2n-1}\{N(\beta_j)\;|\; \beta_j < \tau_1\; {\rm or}\; \beta_j < \tau_2; {\dim \beta_j} = n-2\}\,
\end{equation}
where $N(\beta_j)$ denotes the number of $n$-simplices $\alpha$, such that $\beta_j < \alpha$.

This definition of Ricci curvature is a {\it combinatorial defect} one\footnote{presumably inspired by the classical definition of Gauss curvature as the angular defect at a vertex --- see, e.g. [HC-V].}. This is evident from its expression, but made more transparent by the 2-dimensional case: Indeed, in this case, the simplices $\beta_j$ are 0-dimensional, i.e. vertices, and $N(\beta_j)$ is just the number of 2-simplices having $\beta_j$ as a common vertex, hence ${\rm Ric}(\sigma,\tau_1-\tau_2)$ represents nothing but the total combinatorial defect at these $2n-1$ vertices.

In consequence, using the approach of the original proof of Hilbert and Cohn-Vossen \cite{HC-V}, (and following methods well established in Graphics, etc.), we can consider, instead of the combinatorial defect, the angular defect of the cell $\mathfrak{c}_j$ dual to the vertex $\beta_j$. This, of course, applies both for our way -- as well as Stone's -- of determining a direction.

However, this approach to the  definition of $PL$ Ricci curvature is far less intuitive (and apparently has lesser geometric content, so to speak)
in dimension $\geq 3$. This is the reason why, for our present study, we have made use of the dual complex.
\end{rem}

\begin{rem}
Note that, up to this point, we have not yet defined the sectional curvature $K(\mathfrak{c})$ of a cell $\mathfrak{c}$ (see however the discussion below). Nevertheless, regardless of the specific definition employed, we obtain, quite trivially, the following generalization of the classical curvature bounds of Riemannian geometry (compare also with \cite{Bernig1}, Theorem 1):

\begin{thm}[Comparison theorem] \label{thm:ComparisonThm}
Let $M = M^n_{PL}$ be an $n$-dimensional $PL$ manifold, such that $K_W(M) \geq K_0 > 0$, i.e. $K(\mathfrak{c}) \geq K_0$, for any 2-cell of the dual manifold (cell complex) $M^*$. Then

\begin{equation} \label{eq:comp1}
K_W \lesseqqgtr K_0 \Rightarrow {\rm Ric}_W \lesseqqgtr nK_0\,.
\end{equation}

Moreover

\begin{equation} \label{eq:comp2}
K_W \lesseqqgtr K_0 \Rightarrow {\rm scal}_W \lesseqqgtr n(n+1)K_0\,,
\end{equation}
where ${\rm scal}_W$ denotes the scalar metric curvature of $M$, defined as ${\rm scal}_W(v) = \sum K_W(\mathfrak{c})$, the sum being taken over all the cells of $M^*$ incident to the vertex $v$ of $M^*$.
\end{thm}
\end{rem}

\begin{rem} \label{eq:comp3}
Note that inequality (\ref{eq:comp2}) can be formulated in the seemingly weaker form:

\begin{equation} \label{eq:comp3}
{\rm Ric}_W \lesseqqgtr nK_0 \Rightarrow {\rm scal}_W \lesseqqgtr n(n+1)K_0\,,
\end{equation}

\end{rem}

\begin{rem} \label{rem:ComparisonThm}
Note that in all the inequalities above, the dimension $n$ appears instead of $n-1$ as in the smooth, Riemannian case (hence, for instance one has in (\ref{eq:comp2}), $n(n+1)K_0$, instead of $n(n-1)K_0$\footnote{but even if $n=3$!...} as in the classical case). This is due to our definition (\ref{eq:RicciCell}) of Ricci (and scalar) curvature, via the dual complex of the given triangulation, hence imposing standard and simple combinatorics, at the price of allowing for only for such weaker bounds.\footnote{without affecting the analogue of the Bonnet-Myers Theorem -- see Section 3 below.}
\end{rem}

To determine -- using solely metric considerations -- the sectional curvatures $K(\mathfrak{c}_i)$ of the cells $\mathfrak{c}_i$, we shall employ the so called ({\it modified}) {\it Wald curvature} $K_W$ ($K_W'$). At this point, we have to remind the reader a number of definitions and results that, unfortunately, are perhaps (at least partly) forgotten. We begin with following basic

\begin{defn}
Let $(M,d)$ be a metric space, and let $Q = \{p_1,...,p_4\} \subset M$, together with the mutual distances:
$d_{ij} = d_{ji} = d(p_i,p_j); \, 1 \leq i,j \leq 4$. The set $Q$ together with the set of distances
$\{d_{ij}\}_{1\leq i,j \leq 4}$ is called a {\it metric quadruple}.
\end{defn}

\begin{rem}
One can define metric quadruples in a somewhat more abstract manner, that is without the aid of the ambient
space: In this approach, a metric quadruple is defined as a $4$ point metric space; i.e. $Q = \big(\{p_1,...,p_4\}, \{d_{ij}\}\big)$, where
the distances $d_{ij}$ verify the axioms for a metric.
\end{rem}

We next introduce some necessary notation: Let $S_{\kappa}$ denote the
complete, simply connected surface of constant Gauss curvature $\kappa$, i.e. $S_{\kappa} \equiv \mathbb{R}^2$,
if $\kappa = 0$; $S_{\kappa} \equiv \mathbb{S}^2_{\sqrt{\kappa}}$\,, if $\kappa
> 0$; and $S_{\kappa} \equiv \mathbb{H}^2_{\sqrt{-\kappa}}$\,, if $\kappa < 0$. Here $S_{\kappa} \equiv
\mathbb{S}^2_{\sqrt{\kappa}}$ denotes the sphere of radius  $R = 1/\sqrt{\kappa}$, and $S_{\kappa} \equiv
\mathbb{H}^2_{\sqrt{-\kappa}}$ stands for the hyperbolic plane of curvature $\sqrt{-\kappa}$, as represented by
the Poincar\'{e} model of the plane disk of radius $R = 1/\sqrt{-\kappa}$\,. Using this notation we can next bring

\begin{defn}
 The {\em embedding curvature} $\kappa(Q)$ of the metric quadruple $Q$ is defined to be the curvature $\kappa$ of the gauge surface
$S_{\kappa}$ into which $Q$ can be isometrically embedded.
\end{defn}

We are now able to bring the definition of {\it Wald curvature} \cite{Wa} (or rather of its modification due to Berestovskii \cite{Ber}):

\begin{defn} \label{def:WBcurv}
Let $(X,d)$ be a metric space. 
An open set $U \subset X$ is called a {\it region of curvature} $\geq \kappa$ iff any metric quadruple can be isometrically embedded in $S_m$, for some $m \geq k$.\footnote{While is not needed in the remainder of the paper, we mention for the sake of completeness, that a metric space $(X,d)$ is said to have {\it Wald-Berestovskii curvature} $\geq \kappa$ iff for any $x \in X$ is contained in a region $U$ of curvature $\geq \kappa$.}
\end{defn}

\begin{rem}
Evidently, one can consider the Wald-Berestovskii curvature at an accumulation point of a metric space, hence on a smooth surface, by considering limits of the curvatures of (nondegenerate) regions of diameter converging to 0.
\end{rem}

Before we proceed further, let us make a certain modification of the notation, in order to make it more uniform and more familiar to the reader working in classical Differential Geometry as well as in Graphics: Henceforth we shall denote by $K_W$  the Wald curvature of a surface ($PL$ or smooth), by analogy to its classical (Gauss) curvature $K$. (Of course, $K_W(p)$ will denote the Wald curvature of a point on the surface.)


At this point the question that naturally rises is whether it is possible to actually compute Wald curvature and, if possible, in what manner? First of all, the first, basic step is to note that the role of the abstract open sets $U$ in Definition \ref{def:WBcurv} above is naturally played by the cells $\mathfrak{c}_i$. We can state this as a formal definition, for the record:

\begin{defn}
Let $\mathfrak{c}$ be a cell with vertex set $V_{\mathfrak{c}} = \{v_1,\ldots,v_p\}$. The {\it embedding curvature} $K(\mathfrak{c})$ of $\mathfrak{c}$ is defined as:
\begin{equation}
K(\mathfrak{c}) = \min_{1\leq i<j<k<l\leq p }\kappa(v_i,v_j,v_k,v_l)\,.
\end{equation}
\end{defn}

It is certainly worthwhile to note that it is possible to actually compute the Wald curvature of each of these cells, using the following formula for the embedding curvature $\kappa(Q)$ of a metric quadruple $Q$:

%
%

 \begin{equation} \label{eq:k(Q)}
 \kappa(Q) = \left\{
         \begin{array}{clclcrcr}
           \mbox{0} &  \mbox{if $\Gamma(Q) = 0$\,;} \\
           \mbox{$\kappa,\, \kappa < 0$} & \mbox{if $det({\cosh{\sqrt{-\kappa}\cdot d_{ij}}}) = 0$\,;} \\
           \mbox{$\kappa,\, \kappa > 0$} & \mbox{if $det(\cos{\sqrt{\kappa}\cdot d_{ij}})$ and $\sqrt{\kappa}\cdot d_{ij} \leq
           \pi$}\\
           & \mbox{\,\, and all the principal minors of order $3$ are $\geq 0$;}
         \end{array}
   \right.
\end{equation}
where $d_{ij} = d(p_i,p_j), 1 \leq i,j \leq 4$, and $\Gamma(Q) = \Gamma(p_1,\ldots,p_4)$ denotes, the Cayley-Menger determinant:
\begin{equation} \label{eq:CayleyMenger-nD}
 \Gamma(p_0,\ldots,p_3) = \left| \begin{array}{cccc}
                                            0 & d_{01}^{2} & \cdots & d_{13}^{2} \\
                                            d_{10}^{2} & 0 & \cdots & d_{13}^{2} \\
                                            \vdots & \vdots & \ddots & \vdots \\
                                            d_{30}^{2} & d_{31}^{2} & \cdots & 0
                                      \end{array}
                               \right| \,.
\end{equation}

\begin{rem}
\begin{enumerate}
\item For some first numerical results regarding the application of these formulas in a practical context, see \cite{Sa04}, \cite{SA}. However, it should be noted that, apart from the Euclidean case, the equations involved are transcendental, and can not be solved, in general, using elementary methods.

\item We have also employed Wald curvature as a malleable tool in conjunction with Ricci curvature in a somewhat more theoretical context in \cite{Sa11}.  On a more abstract note, we should remark that, given its (metric) intrinsic nature, $K_W$ ``behaves well'', so to speak, under Gromov-Hausdorff convergence (see \cite{BBI}, \cite{Gr-carte} and \cite{Sa04}, \cite{SA} for some applications in Graphics, Imaging, etc.). Moreover since it (or, rather a somewhat modified version of it $K_{W'}$ identifies with Rinow curvature (see \cite{Bl}, \cite{BM}), it allows us to view the whole problem of defining and computing Ricci for $PL$ (polyhedral) manifolds, (and in particular its applications in Graphics, Regge calculus, etc.) in the larger context of Alexandrov spaces (see, e.g. \cite{BBI}, \cite{Gr-carte}).
\end{enumerate}
\end{rem}

\begin{rem}
Obviously, one can use the same method as above to compute the Ricci curvature (of $\mathcal{T}^\ast$), according to Stone's original approach for determining a directions in cell complexes.
\end{rem}

To return to the main problem of this section:
From the definitions and results above we obtain -- first discretely, at finite scale bounded away from zero -- then passing to the limit) the following result connecting between the Ricci curvatures of a simplicial (polyhedral) complex and its dual:

\begin{thm} \label{ref:RicT,RicT*}
Let $\mathcal{T}, \mathcal{T^\ast}$ be as above. Then
\begin{equation}
\lim_{{\rm mesh}(\mathcal{T}) \rightarrow 0}{\rm Ric}(\sigma) = \lim_{{\rm mesh}(\mathcal{T}^\ast) \rightarrow 0}C\cdot {\rm Ric}^*(\sigma^\ast)\,,
\end{equation}
where $\sigma \in \mathcal{T}$ and where  $\sigma^\ast \in \mathcal{T}^\ast$ is (as suggested by the notation) the dual of $\sigma$.
\end{thm}

\begin{rem}
This result is, admittedly,  somewhat vague. 
However, to our defense, we can only underline the fact that the precise constant $C$ is hard to determine. The thickness condition, that ensures a metric ``quasi-regularity'' of the triangulation, supplies us only with weak estimates. To obtain stronger ones, one should be able to control the regularity of the combinatoric structure, as well. (This is evident, but it will become even clearer in the sequel.) It should be noted in this context that, at least in Graphics, mesh improvement techniques allow us to consider such ``combinatorial almost regular'' triangulations.
\end{rem}


From Lemma \ref{ref:T,T*}, the fact that ${\rm Ric}(v)$ is defined in a purely metric, intrinsic manner and from the fact that intrinsic properties are preserved under Gromov-Hausdorff limits (see \cite{Gr-carte}) and from Theorem \ref{ref:RicT,RicT*} above, we easily  obtain:

\begin{thm}
Let $M^n$ be a (smooth) Riemannian manifold and let $\mathcal{T}$ be a thick triangulation of $M^n$. Then
\begin{equation}
{\rm Ric}_\mathcal{T} \rightarrow C_1\cdot{\rm Ric}_{M^n},\; {\rm as}\; {\rm mesh}(\mathcal{T}) \rightarrow 0 \,,
\end{equation}
where the convergence is the weak convergence (of measures). 
\end{thm}
For related results, see \cite{cms}, \cite{Sa11a}  for the Lipschitz-Killing curvatures, \cite{BK}, \cite{Sa11}  for discrete (combinatorial, respective metric) Gaussian curvature, and \cite{Bernig2}, for the Einstein measures.

\begin{rem}
While the desired constant $C_1$ is, of course, $C_1 = 1$, and some first experimental results hint that, at least for certain ``nice'' triangulations, this is indeed the case,
 we can't guarantee a better result -- see the remark following the preceding theorem.
\end{rem}

%

\begin{rem}
While we have adopted the Wald curvature as the metric curvature for surfaces\footnote{and the Finsler-Haantjes one as a metric alternative for computation of principal curvature} of our choice, for reasons detailed above, it would be interesting to explore the capabilities  -- both theoretical and practical -- as far as $PL$ Differential Geometry is concerned, of other metric curvatures (see \cite{Sa11a} and the bibliography therein)  
and in particular of the {\it Menger curvature measure}:
\begin{equation}
\mu(\mathcal{T}) = \mu_p(\mathcal{T}) = \sum_{T \in \mathcal{T}}\kappa_M^p(T)({\rm diam}\,T)^2\,,
\end{equation}
for some $p \geq 1$, where $\kappa_M$ denotes the Menger curvature (of the simplex $T$).
\end{rem}


\section{The Bonnet-Myers Theorem}
Having introduced a metric Ricci curvature for $PL$ manifolds, one naturally wishes to verify that this represents, indeed, a proper notion of Ricci curvature, and not just an approximation of the classical notion. According to the synthetic approach to Differential Geometry (see, e.g.  \cite{Gr-carte}, \cite{Vi}), a proper notion of Ricci curvature should satisfy adapted versions of the main, essential theorems that hold for the classical notions. Amongst such theorems the first and foremost is Myers' Theorem (see, e.g., \cite{Be}). And, indeed, fitting versions for combinatorial cell complexes and weighted cell complexes were proven, respectively, by Stone \cite{St1}, \cite{St2}, and Forman \cite{Fo}. Moreover, the Bonnet part of the Bonnet-Myers theorem, that is the one appertaining to the sectional curvature, was also proven for $PL$ manifolds, again by Stone -- see \cite{St3}, \cite{St0}.

For the special -- yet of main importance in applications (see \cite{CL}, \cite{GY}, \cite{Sa11}) -- case of 2-dimensional manifolds, such a result is easy to prove, given the fact that Ricci and sectional curvature essentially coincide. More precisely, we can formulate the following theorem:

\begin{thm}[Bonnet-Myers for $PL$ 2-manifolds -- Combinatorial] \label{thm:BM-Comb}
Let $M^2_{PL}$ be a complete, closed 2-dimensional $PL$ manifold of without boundary, such that

(i) There exists $d_0 > 0$, such that ${\rm mesh}(M^2_{PL}) \leq d_0$\footnote{Here ${\rm mesh}(M^2_{PL})$ denotes the mesh of the 1-skeleton of $M^2_{PL}$, i.e. the supremum of the edge lengths.}

(ii) $K_{Comb}(M^2_{PL}) \geq K_0 > 0$.

Then $M^2_{PL}$ is compact and, moreover
\begin{equation} \label{eq:estimate}
{\rm diam}(M^2_{PL}) \leq \left\{
\begin{array}{ll}
2\pi d_0, & k_0 \geq (2 - \sqrt{2})\pi\,;\\
4\pi^3d_0/[(2\pi - d_0)(4\pi k_0 - k_0^2)^{1/2}], & {\rm else}\,;
\end{array}
\right.
\end{equation}
where $K_{Comb}$ denotes the combinatorial Gauss curvature of $M^2_{PL}$,
\begin{equation}
K_{Comb}(p) = 2\pi - \sum_{i=1}^{m_p}\alpha_i(p)\,
\end{equation}
where $\alpha_1,\ldots,\alpha_{m_p}$ are the (interior) face angles adjacent to the vertex $v_i$.
\end{thm}

\begin{rem}
Condition (1), that ensures that the set of vertices (of the $PL$ manifold) is ``fairly dense''\footnote{in Stone's formulation (\cite{St0}, p. 1062).} is nothing but the necessary and quite common density condition for good approximation both of distances and of curvature measures -- see e.g. \cite{cms} and \cite{Sa11a} and the references therein. The mere existence of such a $d_0$ is evident for a compact manifold, however it can't be apriorily be supposed for a general manifold, hence has do be postulated. Moreover, to ensure a good approximation of curvature, this density factor has to be properly chosen (see, e.g. \cite{SAZ}), thus tighter estimates for the mesh of the triangulation can be obtained from (\ref{eq:estimate}) along with better curvature approximation. No less importantly, an adequate choice of the vertices of the triangulation, also ensures, via the thickness property, the non-degeneracy of the manifold (and of its curvature measures) -- see \cite{Sa11a}.
\end{rem}

\begin{prf1}
The theorem follows readily from Theorem 3 of \cite{St3}. Indeed, in the two dimensional case, the so called {\it maximum} and {\it minimum curvatures}, $k_+$, respective $k_-$ (see \cite{St3}, p. 12, for the precise definitions) at the vertices of $M^2_{PL}$ coincide with the combinatorial Gauss curvature.
Moreover, conditions (1) and (2) of Theorem 3 of \cite{St3} are, due to the fact that here we are concerned solely with 2-dimensional simplicial complexes ($PL$ manifolds), equivalent to our conditions (2) and (1) above, respectively.
Therefore, the conditions in the statement of Theorem 3, \cite{St3} are satisfied and, by (ii) of the said result, the theorem above follows immediately. 
\end{prf1}

\begin{rem}
It is easy to see that the theorem above extends to more general polyhedral surfaces. Indeed, by their very definition
such surfaces admit simplicial subdivisions. 
However, during this subdivision, $k_+$, respective $k_-$ do not change, since the only relevant contributions to these quantities occur at the vertices, and depend only on the angles at these vertices, more precisely on the {\it normal geometry} (see \cite{St3}, p. 12), that suffer no change during the subdivision process.
\end{rem}

\begin{rem}
The bound (\ref{eq:estimate}) is rather weak, as compared to the one for the classical case, but it is the only one supplied by Stone's result we made appeal to, namely Theorem 3 of \cite{St3}.
\end{rem}

The proof above suffers from the disadvantage of making use of Stone's maximum and minimum curvatures (even though, in this context making appeal to them is rather natural). We can, however, provide a different proof, independent of Stone's work, but at the price of using some heavy (albeit classical) machinery, that, moreover, takes us away, so to say, from the discrete methods. (On the other hand, smooth, analytical tools are far more familiar to a large research community in CAGD, Imaging, etc.)\\

\begin{prf2}
The basic idea (which we first employed in \cite{Sa11}) is to consider a {\it smoothing} $M^2$ of $M^2_{PL}$.
Since, by \cite{Mun}, Theorem 4.8, smoothings
approximate arbitrarily well both distances and angles\footnote{More precisely, they are $\delta$-approximation and, for $\delta$ small enough, also $\varepsilon$-approximations of $M^2_{PL}$ -- for details see \cite{Mun}, or, just for the minimal required facts, the Appendix of \cite{Sa11}.} on $M^2_{PL}$, defects are also arbitrarily well approximated. Given that the combinatorial curvature of $M^2_{PL}$ is bounded from below, it follows that so will be the sectional (i.e. Gauss) curvature of $M^2$.


Unfortunately, the Gaussian curvature of $M^2$ is positive only on isolated points (the set of vertices of $M^2_{PL}$), so we can not apply the classical Bonnet theorem yet. However, we can ensure that $M^2$ is arbitrarily close to a smooth surface $M^2_+$, having curvature Gaussian curvature $K(M^2_+) > 0$.\footnote{This is easily seen by adding spherical ``roofs'' (of low curvature) over the faces, and then slightly modifying the construction, to ensure that the curvature will be positive also on the ``sutures'' of the said roofs, corresponding to the edges of the original $PL$ manifold.}
Therefore, the classical Bonnet Theorem can be applied for $M^2_+$, hence $M^2_{PL}$ is compact and its diameter has the same upper bound (again using the same arguments as before\footnote{i.e. $\delta$- and $\varepsilon$-approximations}) as that of $M^2_+$ (and $M^2$), namely
\begin{equation} \label{eq:BM}
{\rm diam}(M^2_{PL}) \leq \frac{\pi}{\sqrt{K_0}}\;.
\end{equation}
\end{prf2}

\begin{rem}
Apparently, the bound for diameter given by the proof above, is tighter than the one obtained by Stone in \cite{St3}. However, we should keep in mind that, in practice, one is more likely to encounter $PL$ surfaces as approximations of smooth ones.\footnote{and, obviously, $PL$ surfaces are $PL$ approximations of their own smoothings} However, the larger the mesh of the approximating surface (i.e. the ``rougher'' the approximation), the larger the deviation of the approximating triangles from the tangent planes (at the vertices), hence the more likely is to obtain large combinatorial curvature. Hence, there is a correlation between size of the simplices and curvature, more precisely, the lower bounds in (\ref{eq:estimate}), the lower ones in (\ref{eq:BM}).
\end{rem}

Since the leitmotif of the previous section was metric (Wald) curvature, it is natural to ask whether a fitting version of the Bonnet-Meyers Theorem exists for this type of curvature? The answer is -- at least in dimension 2 -- positive: we can, indeed state an analogue of Meyers' Theorem, in terms of the Wald curvature:

\begin{thm}[Bonnet-Meyers for $PL$ 2-manifolds -- Metric] \label{thm:BM-Metric}
Let $M^2_{PL}$ be a complete, 2-dimensional $PL$ manifold without boundary, such that

(i') There exists $d_0 > 0$, such that ${\rm mesh}(M^2_{PL}) \leq d_0$;

(ii') $K_W(M^2_{PL}) \geq K_0 > 0$.

Then $M^2_{PL}$ is compact and, moreover
%
\begin{equation} 
{\rm diam}(M^2_{PL}) \leq \frac{\pi}{\sqrt{K_0}}\;.
\end{equation}
\end{thm}

\begin{proof}
We employ again the basic argument first used in \cite{Sa11}:
Since distances (and angles) are arbitrarily well approximated by smoothings, it follows that so are metric quadruples (including their angles), hence so is Wald curvature. By \cite{Bl} (see also [BM70], Theorems
11.2 and 11.3), the Wald curvature at any point of non-trivial geometry $M^2$, namely at a vertex $v$,
$K_W(v)$ equals the classical (Gauss) curvature $K(v)$ (and, of course, this is also true a fortiori at all the other points, where both the smooth and the $PL$ manifold are flat). Therefore the Gauss curvature of $M^2$ approximates arbitrarily well 
the Wald curvature of $M^2_{PL}$, hence we can apply the same argument as in Proof 2 above to show that $M^2_{PL}$ is, indeed, compact and, furthermore, satisfies the upper bound (\ref{eq:BM}).
%
%
\end{proof}

\begin{rem}
Like the previous theorem, the result above can be extended to polyhedral manifolds, and even in a more direct fashion, since Wald curvature does not take into account the number of sides of the faces incident to a vertex, but only their lengths.
\end{rem}

\begin{rem}
This result, as well as its generalization to higher dimensions (see \ref{thm:BM-Metric+}) is hardly surprising, given the fact that, by \cite{BGP}, Theorem 3.6, Myers' theorem holds for general Alexandrov spaces of curvature $\geq K_0 > 0$, and since Wald-Berestovskii curvature is essentially equivalent to the Rinow curvature (see \cite{Bl}), hence to the Alexandrov curvature (see, e.g. \cite{Gr-carte}, Chapter 1). Rather, we give, in the special case of $PL$ surfaces (manifolds) a simpler, more intuitive 
proof of the Burago-Gromov-Perelman extension of Meyers' Theorem.
\end{rem}

In higher dimension, none of the arguments applied in both proofs of Theorem \ref{thm:BM-Comb} are applicable, at least not without imposing further conditions:

\begin{itemize}

\item Regarding the first proof:

\begin{itemize}

\item In dimensions higher than 2, $k_+$ and $k_-$ do not, necessarily equal each other (see \cite{St3}, Example 4, p. 14)
 and, a fortiori, they fail to equal the combinatorial Gauss curvature. They do, however, according to Stone \cite{St3}, resemble in their behaviour the minimum, respective maximum sectional curvature at a point common to two 2-planes, that contain a given (fixed) tangent vector at the point in question.

 An important proviso should be added, however: While for the general $PL$ simplicial complexes, the equality between $k_+$ and $k_-$ fails to hold, it is true for the most relevant -- at least as far as our analysis is concerned -- case of $PL$ manifolds without boundary (see \cite{St3}, Example 3, p. 13).
Consequently, it is not clear how to connect our proposed metric discretization of Ricci curvature with the the maximal and minimal curvatures of Stone (hence to combinatorial curvature, whenever they equal it -- and each other).\footnote{A natural attempt would be to use straightforward extensions of $k_+$ and $k_-$ -- let's denote them, for convenience,
 ${\rm Ric}_{\min}$ and  ${\rm Ric}_{\max}$. However, it is not clear (at least at this point in time) how expressive  these definitions would prove to be.}

 \begin{rem}
 It is true that the lower bound on $k_+$, as considered in Theorem 3 of \cite{St3} has a simple expression, in any dimension, via a topological condition (cf. Lemma 5.1 of \cite{St3}), namely that the intersection of any ($PL$) geodesic segment of ends $p$ and $q$  with the 2-skeleton of $M^2_{PL}$ is precisely the set $\{p,q\}$ (with the exception, of course, of the case when the segment is contained in a simplex (of $M^2_{PL}$). However, since the metric information contained in this new condition is void (or rather thoroughly encrypted, so to say) it has no apparent advantage for application in conjunction with metric curvature.
 \end{rem}

\item For an application of the Stone's methods in conjunction with the metric curvature approach to any dimension, one would have to make appeal to Jacobi fields, as defined in \cite{St1}. However, as discussed in the previous section, this would probably led to numerical instability.

\end{itemize}

\item As far as the second proof is concerned:

\begin{itemize}

\item No smoothing of a $PL$ manifold necessarily exists in dimension higher than 4 and, even if it exists, it is not necessarily unique (starting from dimension 4)  -- see \cite{Mun1}.

    However, if such a smoothing exists, then the second proof of Theorem \ref{thm:BM-Comb} (and of Theorem \ref{thm:BM-Metric}) extends to any dimension, and we obtain the following $PL$ (metric) versions of the classical results:

    \begin{thm}[$PL$ Bonnet -- metric]
    \label{thm:BM-Metric+}
    Let $M^n_{PL}$ be a complete, $n$-dimensional $PL$, smoothable manifold without boundary, such that

(i') There exists $d_0 > 0$, such that ${\rm mesh}(M^n_{PL}) \leq d_0$;

(ii') $K_W(M^n_{PL}) \geq K_0 > 0$\,,

where $K_W(M^n_{PL})$ denotes the sectional curvature of the ``combinatorial sections'' i.e. the cells $c_i$ (see Section 1 above).

Then $M^n_{PL}$ is compact and, moreover
%
\begin{equation} 
{\rm diam}(M^2_{PL}) \leq \frac{\pi}{\sqrt{K_0}}\;.
\end{equation}
    \end{thm}

%

    In all honesty, we should add that the ``rounding'' argument of Proof 2 of Theorem \ref{thm:BM-Comb} is not easy to extend directly -- if at all -- to higher dimension. Instead, a more subtle argument has to be devised. To this end we make appeal again to Stone's paper \cite{St3}, and we build the spherical simplicial complex $M^n_{Sph,\rho}$ associated to the given $PL$ (or rather piecewise-flat) complex $M^n_{PL}$. This is built as follows: Consider the sphere of radius $R = R(\sigma)$ and radius $O = O(\sigma)$, circumscribed to a given simplex $\sigma$, and its image $\sigma^* = \sigma^*(R^*)$ on a sphere of radius $R^* = R^*(\sigma), R^* \geq R$, via the central projection from $O$. Then we denote by   simplicial complex obtained by remetrization of $M^n_{PL}$ by the replacement of each $\sigma$ by its spherical counterpart $\sigma^*$. Then, by Lemma 5.5 of \cite{St3}, for large enough $R^* > R$, the following holds for any pair of points $p,q \in M^n_{PL}$:
     ${\rm dist}_{M^n_{PL}}(p,q) \leq C{\rm dist}_{M^n_{Sph,\rho}}(p^*,q^*)$, for a certain constant $C$, where $p^*,q^*$ denote the spherical images of $p,q$. Since the curvature at each vertex of the spherical simplex obtained by central projection of the simplices of $M^n_{PL}$ onto their circumscribed spheres is smaller than the corresponding one (at the same vertex) in the $PL$ (piecewise flat manifold), this holds a fortiori for $M^n_{Sph,\rho}$\,. It follows from the classical Bonnet theorem (after applying the necessary smoothing) that ${\rm diam}(M^n_{PL}) < {\rm diam}(M^n_{Sph,\rho})$.

\item On the other hand, if we approach the problem of $PL$ Ricci curvature from the viewpoint of the first part of the paper, that is of $PL$ approximations of smooth manifolds, then the situation changes dramatically. Indeed, even when such a smoothing $M^n$ ($n \geq 3$) exists, it is not probable that its sections provided in this manner by $M^n_{PL}$ suffice to approximate well enough -- let alone reconstruct -- the Ricci curvature of $M^n$. In simple words, ``there are not enough directions'' in $M^n_{PL}$ to allow us to infer from the metric curvatures of a $PL$ approximation, those of a given smooth manifold $M^n$ (in fact, not not even a good approximation), hence we are faced again with a problem that we already mentioned in conjunction with the first proof, namely that of insufficient ``sampling of directions'' in $PL$ approximations. (On the other hand, increasing of the number of directions, i.e. of 2-dimensional sections (simplices) generates a decrease of the the precision of the approximation, due to the (possible) loss of thickness of the triangulation -- a problem which we have discussed in some detail in \cite{Sa11a}.)

\begin{rem}
The considerations above show us that, unfortunately, no analogue in higher dimensions of the Myers' Theorem can be obtained by applying smoothing arguments). It is true that {\it a} Ricci curvature of the smooth manifold $M^n$ is obtained in terms of that of $M^n_{PL}$, however, it is not clear, in view of the paucity of sectional directions (i.e. possible 2-sections), how precisely is this connected to its discrete counterpart. Therefore, we can obtain, at best, an approximation result (with limits imposed by the thickness constraint -- see discussion above).
\end{rem}

\end{itemize}

\end{itemize}

We conclude with the following remarks: From the discussion above is transparent that, unfortunately, at this point in time, we can offer no proof for the general case, that is for non-smoothable $PL$ manifolds of dimension $n \geq 4$. To obtain such a proof for Bonnet's Theorem, one should adapt Stone's methods, as developed in \cite{St3}, while for a comprehensive generalization of Myers' theorem, one has the apparently more difficult task of adapting the purely combinatorial methods of \cite{St1} to the metric case. A quite different approach, but one that would allow us to extend the metric approach to quite general weighted $CW$ complexes, would be to adapt Forman's methods developed in \cite{Fo} to our case. The essential step in this direction would be to find relevant geometric content (e.g. lengths, area, volume) for Forman's ``standard weights'' associated to each cell.



\end{document}